\documentclass[12pt]{article}
\usepackage{amsmath, amsthm, amscd, amsfonts, amssymb, graphicx, color}
\usepackage[bookmarksnumbered, colorlinks, plainpages]{hyperref}
\usepackage[utf8]{inputenc}
\newtheorem{thm}{Theorem}[section]
\newtheorem{cor}[thm]{Corollary}
\newtheorem{lem}[thm]{Lemma}
\newtheorem{prop}[thm]{Proposition}
\newtheorem{defn}[thm]{Definition}

\numberwithin{equation}{section}

\newcommand{\be}{\begin{equation}}
\newcommand{\ee}{\end{equation}}
\newcommand{\ben}{\begin{enumerate}}
\newcommand{\een}{\end{enumerate}}
\newcommand{\beq}{\begin{eqnarray}}
\newcommand{\eeq}{\end{eqnarray}}
\newcommand{\beqn}{\begin{eqnarray*}}
\newcommand{\eeqn}{\end{eqnarray*}}

\newcommand{\bpf}{\begin{proof}}
\newcommand{\epf}{\end{proof}}
\newcommand{\bl}{\begin{lem}}
\newcommand{\el}{\end{lem}}
\newcommand{\bp}{\begin{prop}}
\newcommand{\ep}{\end{prop}}
\newcommand{\bd}{\begin{defn}}
\newcommand{\ed}{\end{defn}}
\newcommand{\bt}{\begin{thm}}
\newcommand{\et}{\end{thm}}

\title{Complete Ricci solitons on Finsler manifolds}
\author{B. Bidabad\footnote{Corresponding author; bidabad@aut.ac.ir} \  and  M. Yar Ahmadi}
\date{\small Faculty of Mathematics and Computer Science\\
Amirkabir University of Technology (Tehran Polytechnic).}

\begin{document}
\maketitle
\begin{abstract}
The geometric flow theory and its applications turned into one of the most intensively developing branches of modern geometry.
Here, a brief introduction to Finslerian  Ricci flow and their self-similar solutions known as   Ricci solitons are given and some recent results are presented.
They are a generalization of  Einstein metrics and
 are previously developed by the present authors for Finsler manifolds.
In the present work, it is shown that a complete shrinking Ricci soliton Finsler manifold has a finite fundamental group.
\end{abstract}
\section{Introduction}
The Ricci flow, which is an evolution equation for a Riemannian metric in the set of all Riemannian metrics on a given manifold,
is an important example of a geometric flow.
One can find a geometric flow
which deforms a metric to a canonical metric, or at least one which can be used to deduce topological
information about the underlying manifold. However, general results of this type are only known in
low dimensions. Even then, it is not true that the Ricci flow deforms a metric on a three-manifold to a ``nice" metric, at least not without first passing through a number of surgeries.\\

\textbf{Case of Riemannian geometry;} In 1982 Hamilton introduced the notion of Ricci flow on Riemannian manifolds by the evolution equation
\begin{align}\label{RRF}
\frac{\partial}{\partial t}g_{ij}=-2Ric_{ij}, \quad g(t=0):= g_0.
\end{align}
The Ricci flow, which evolves a Riemannian metric by its Ricci curvature is a natural analog of the heat equation for metrics. In Hamilton's celebrated paper \cite{Ha}, it is shown that there is a unique solution to the Ricci flow for an arbitrary smooth Riemannian metric on a closed manifold over a sufficiently short time.

The concept of Ricci solitons was introduced by Hamilton in 1988, as self-similar solutions to the Ricci flow, see \cite{Ha1}. They are natural generalizations of Einstein metrics called also Ricci solitons and are subject to a great interest in geometry and physics especially in relation to string theory.

J. Lott has shown that the fundamental group of a closed manifold $M$ is finite for any gradient shrinking Ricci soliton, see \cite{Lott}.
As a result of A. Derdzinski,
every compact shrinking Ricci soliton has only finitely many free homotopic classes of closed curves in $M$ that  are in a bijective correspondence with the conjugacy classes in the fundamental group of $M$, see \cite{Derdz}. M. F. L\'{o}pez and E. G. R\'{i}o have proved that a compact shrinking Ricci soliton has finite fundamental group \cite{FG}. Moreover, Wylie has shown that a complete shrinking Ricci soliton has finite fundamental group \cite{Wylie}. An extension of similar results is obtained for Yamabe solitons by the present authors in \cite{YB3}. \\

\textbf{Case of Finsler geometry;} The concept of Ricci flow on Finsler manifolds is defined first by D. Bao,  \cite{bao}, choosing the Ricci tensor introduced by H. Akbar-Zadeh, see for instance \cite{Ak3}. Recently, the existence and uniqueness of solutions to the Ricci flow on Finsler surfaces are shown by the first present author in a joint work, cf. \cite{BS}.
 Ricci solitons in Finsler manifolds, as a generalization of Einstein spaces, are introduced by the present authors and it is shown that if there is a Ricci soliton on a compact Finsler manifold then there exists a solution to the Finslerian Ricci flow equation and vice-versa, see \cite{BY}.
Next,  a Bonnet-Myers type theorem was studied and it is proved that on a Finsler space, a forward complete shrinking Ricci soliton space is compact if and only if the corresponding vector field is bounded. Moreover, it is proved that a compact shrinking Ricci soliton Finsler space has a finite fundamental group and hence the first de Rham cohomology group vanishes, see \cite{YB1}.\\

 In the present work,  the results on shrinking Ricci soliton Finsler spaces previously obtained in the compact case by the present authors  \cite{YB1}, are extended for geodesically complete spaces, which extends also a result of Wylie, cf. \cite{Wylie}. The following theorem applies to a more general class of Finsler manifolds than Ricci solitons.
  \begin{thm}\label{theorem1}
 Let $(M,F)$ be a  geodesically complete Finsler manifold satisfying
\begin{align}\label{Eq;a}
2F^2\mathcal{R}ic+\mathcal{L}_{\hat{V}}{F^2}\geqslant 2 \lambda F^2,
\end{align}
where,  $\lambda > 0$. Then
  \begin{align}\label{theorem1bound}
  d(p,q)\leqslant \max\big\{1,\frac{1}{\lambda}\big(2(n-1)+H_p+H_q+\Vert V\Vert_p+\Vert V\Vert_q\big)\big\},
  \end{align}
  where $\Vert V\Vert_p$ denotes the``length" of $V$ with respect to $F$.
\end{thm}

The inequality \eqref{theorem1bound} is equivalent to say that the distance $d(p,q)$ defined by the Finsler structure $F$ has an upper bound depending on the geometry in the forward-one-ball and backward-one-ball around $p$ and $q $ in $M$.
As a consequence, we have the following results.
\begin{thm}\label{thm2}
Let $(M,F)$ be a complete Finsler manifold satisfying (\ref{Eq;a}). Then the fundamental group $\pi_1(M)$ of $M$ is finite.
\end{thm}
In particular, it follows that  a complete shrinking Ricci soliton Finsler space has a finite fundamental group and hence its first de Rham cohomology group vanishes.
\section{Preliminaries and terminologies}
 Let $M$ be a real {\it n}-dimensional differentiable  manifold. We denote by $TM$ its tangent bundle and by $\pi :TM_{0} \longrightarrow M$ the bundle of non zero tangent vectors.
A \emph{Finsler structure} on $M$ is a function $F:TM\longrightarrow [0,\infty)$ with the following properties:\\
(i) Regularity: $F$ is $C^{\infty}$ on the entire slit tangent bundle $TM_{0}=TM\backslash 0$.\\
(ii) Positive homogeneity: $F(x,\lambda y)=\lambda F(x,y)$ for all $\lambda > 0$.\\
(iii) Strong convexity: The $n\times n$ Hessian matrix $g_{ij}=([\frac{1}{2} F^{2}]_{y^{i} y^{j}})$ is positive definite at every point of $TM_{0}$.
A \emph{Finsler manifold} $(M,F)$ is a pair consisting of a differentiable manifold $M$ and a Finsler structure $F$. The \emph{Cartan tensor}  the \emph{formal Christoffel symbols} and the \emph{spray coefficients} are defined here by $C_{ijk}=\frac{1}{2}\frac{\partial}{\partial y^{i}}g_{jk}$,
\begin{align}\label{chri}
\gamma^i_{jk}:=g^{is}\frac{1}{2}\big(\frac{\partial g_{sj}}{\partial x^k}-\frac{\partial g_{jk}}{\partial x^s}+\frac{\partial g_{ks}}{\partial x^j}\big),
\end{align}
 and
 \begin{align}\label{E,G}
 G^i&:=\frac{1}{2}\gamma^i_{jk}y^j y^k,
\end{align}
respectively.
We consider also the \emph{reduced curvature tensor} $R^i_k$ which is a connection free quantity and is expressed in terms of partial derivatives of spray coefficients as follows.
 \begin{align}\label{E,Ricci scalar}
 R^{i}_{k}:=\frac{1}{F^2}\big(2\frac{\partial G^i}{\partial x^k}-\frac{\partial^2 G^i}{\partial x^j \partial y^k}y^j +2G^j\frac{\partial^2 G^i}{\partial y^j \partial y^k} - \frac{\partial G^i}{\partial y^j}\frac{\partial G^j}{\partial y^k}\big).
\end{align}


\subsection{Lie derivative on Finsler manifolds}
 Let $V=v^{i}(x) \frac{\partial}{\partial x_i}$  be a vector field on $M$.  If \{$\varphi_t$\} is the local one-parameter group of $M$ generated by $V$, then it induces an infinitesimal point transformation on $M$ defined by $\varphi^\star_t (x^{i}):= \bar{x}^{i}$, where $\bar{x}^i=x^{i}+v^{i}(x) t$, for all $t$ sufficiently close to zero, that is $t\in (-\epsilon,\epsilon)$ and $\epsilon>0$. This is naturally extended to the point transformation $\tilde{\varphi_t}$ on the tangent bundle $TM$ defined by $\tilde{\varphi_t}^{\star} := (\bar{x}^{i},\bar{y}^{i})$, where
\begin{align} \label{E;T}
\bar{x}^{i}=x^{i}+v^{i}(x) t, \quad \bar{y}^{i}=y^{i}+\frac{\partial v^{i}}{\partial x^{m}}  y^{m} t.
\end{align}
It can be shown that, $\{\tilde{\varphi_t}\}$ induces a vector field $\hat{V}=v^{i}(x) \frac{\partial}{\partial x^{i}}+y^{j} \frac{\partial v^{i}}{\partial x^{j}}  \frac{\partial}{\partial y^{i}}$ on $TM$ called the \emph{complete lift} of $V$. The
 one-parameter group associated to  the complete lift $\hat{V}$ is given by  $\tilde{\varphi_t}(x,y)=(\varphi_t(x),y^i \frac{\partial \varphi_t}{\partial x^i})$.

Let $\Upsilon^{I}(x,y)$  be  an arbitrary geometric object on $TM$, where $I$ is a mixed multi index and $\hat{V}$ the complete lift  of a vector field $V$ on $M$.
The \emph{Lie derivative} of  $\Upsilon^{I}(x,y)$ in direction of $\hat{V}$ is defined by
\begin{align} \label{Df;LieDer}
{\mathcal{L}_{\hat{V}} \Upsilon^I=\lim_{t\rightarrow 0}\frac{\tilde{\varphi_t}^{\star}(\Upsilon^{I})-\Upsilon^{I}}{t}}=\frac{d}{dt}\tilde{\varphi_t}^{\star}(\Upsilon^{I}),
\end{align}
where $\tilde{\varphi_t}^\star(\Upsilon^{I})$ is the deformation of $\Upsilon^{I}(x,y)$ under the extended point transformation (\ref{E;T}). Whenever the geometric object $\Upsilon^{I}$ is a tensor field, $\tilde{\varphi_t}^\star(\Upsilon^{I})$ coincides with the classical notation of pullback of $\Upsilon^{I}(x,y)$.
Lie derivative of the Finsler metric tensor $g_{jk}$ is given in the following tensorial form by
\begin{align}\label{Eq;Lieder}
\mathcal{L}_{\hat V}g_{jk}=\nabla_jV_k+\nabla_kV_j+2(\nabla_0V^l)C_{ljk},
\end{align}
where $\nabla$ is the horizontal covariant derivative in Cartan connection, $\nabla_p=\nabla_{\frac{\delta}{\delta x^p}}$ its components and $\nabla_0=y^p\nabla_p$, see \cite{BJ} for more details.
\subsection{The second variation of arc length in Finsler geometry}
Let $(M,F)$ be a Finsler manifold and $\gamma:[a,b]\longrightarrow M$ a piecewise $C^{\infty}$ curve on $M$, with the velocity $\frac{d\gamma}{dt}=\frac{d\gamma^i}{dt}\frac{\partial}{\partial x^i}\in T_{\gamma(t)}M$. Its \emph{Finslerian length} $L(\gamma)$ is defined by
$L(\gamma)=\int_a^bF(\gamma,\frac{d\gamma}{dt})dt.$
For $p,q\in M$, denote by $\Gamma(p,q)$ the collection of all piecewise smooth curves $\gamma:[a,b]\longrightarrow M$ with $\gamma(a)=p$ and $\gamma(b)=q$. Define the \emph{Finslerian distance} function $d:M\times M\longrightarrow [0,\infty)$ by
$d(p,q):=\inf\limits_{\gamma\in\Gamma(p,q)}L(\gamma).$
Note that in general this distance function does not have the symmetric property.
A Finsler manifold $(M, F)$ is said to be\emph{ forward} (resp. \emph{backward}) \emph{geodesically complete} if every geodesic $\gamma(t)$, for $a\leq t<b$ (resp. $a<t \leq b$), parameterized by arc length, can be extended to a geodesic defined on $a\leq t<\infty$ (resp. $\infty <t\leq b$). In general, forward geodesically completeness is not equivalent to backward geodesically completeness, see \cite[p. 168]{bcs}.
 According to the Hopf-Rinow theorem,  on a forward (or backward) geodesically complete Finsler space, every two points $p,q\in M$ can be joined by a minimal geodesic. The forward and backward metric balls $\mathcal{B}^+_p(r)$ and $\mathcal{B}^-_p(r)$ are defined by
$\mathcal{B}^+_p(r):=\{x\in M\ |\ d(p,x)<r\}$ and $\mathcal{B}^-_p(r):=\{x\in M\ |\ d(x,p)<r\}$, respectively.

Here we set up our notations for the second variation of arc length similar to that of in \cite{Ak3} and \cite{bcs}. Let $\gamma:[0,r]\longrightarrow M$ be a regular piecewise smooth curve in $M$. Define
$\mathsf R=\{(s,t)\vert \ 0\leqslant s \leqslant r,\ \ -\epsilon\leqslant t \leqslant \epsilon\}.$
A piecewise smooth \emph{variation} of $\gamma(s)$ is a continuous map $\gamma(s,t)$ from $\mathsf R$ into $M$ which is piecewise smooth and such that $\gamma(s,0)$ reduces to the given $\gamma(s)$. Their velocity fields give rise, respectively, to the two vector fields $T:=\gamma_*\frac{\partial}{\partial s}=\frac{\partial\gamma}{\partial s}$ and $U:=\gamma_*\frac{\partial}{\partial t}=\frac{\partial\gamma}{\partial t}$,  are tangent to $M$.
  The map $\gamma(s,t)$ admits a canonical lift $\hat{\gamma}:\mathsf R\longrightarrow TM_0$ defined by
$$\hat{\gamma}(s,t):=(\gamma(s,t),T(s,t)).$$
Without loss of generality, we can also consider $T$ and $U$ as elements in the fibre of $\pi^*TM$ over the point $\hat\gamma$.
Corresponding to $\hat{\gamma}$, one gets the following vector fields,
$$\hat T:=\hat\gamma_*\frac{\partial}{\partial s}=\hat\gamma^{\prime}(s),\qquad \hat U:=\hat\gamma_*\frac{\partial}{\partial t}=\frac{\partial\hat\gamma}{\partial t},$$
where $\hat T,\hat U\in T_{_{\hat\gamma(s,t)}}TM_0.$

Now, let $\gamma(s)$, $s\in [0,r]$, be a geodesic of Cartan connection parameterized by the arc length $s$. The \emph{second variation} of arc length in Finsler geometry is given by
\begin{align}\label{secondVAL}
 L^{\prime\prime}(0)=g(\nabla_{\hat U}U,T)\Big\vert_0^r+\int_0^r\Big[g(\nabla_{\hat T}U,\nabla_{\hat T}U)-g(R(U,T)T,U)-|\frac{\partial}{\partial s}g(U,T)|^2\Big]ds,
\end{align}
where $R(U,T)T$ is the {\it hh}-curvature of  Cartan connection. For a global approach to the  definition of  {\it hh}-curvature of  Cartan connection one can refer to \cite[p. 218]{Ak3} or \cite[p. 3]{BSY}.
\section{An estimation for the integral of Ricci tensor along geodesics}
 Let $(M,F)$ be a Finsler manifold and $p\in M$. Define
\begin{align}\label{Hpp}
H_p=\sup\limits_{x\in \mathcal{B}^+_p(1)\cup \mathcal{B}^-_p(1)}\max\limits_{y\in S_xM}|\mathcal{R}ic(x,y)|,
\end{align}
where $S_xM:=\{y\in T_xM | F(x,y)=1\}$  is always compact whether $M$ is compact
or not and $H_p $  is bounded. Denote by $SM$ the sphere bundle, defined by $SM:=\bigcup\limits _{x\in M} S_xM$. Now we are in a position to prove the following lemma.
\begin{lem}\label{lemma}
Let $(M,F)$ be a  complete Finsler manifold, $p,q\in M$ such that $r:=d(p,q)>1$ and $\gamma$ a minimal geodesic from $p$ to $q$ parameterized by the arc length $ s $. We have
$$\int_0^r\mathcal{R}ic(\gamma,\gamma^{\prime})ds\leqslant 2(n-1)+H_p+H_q.$$
\end{lem}
\begin{proof}
Let $\{E_i\}_{i=1}^n$ be an orthonormal frame for $\pi^*TM$ along $\gamma(s)$, where $E_n:=\gamma^{\prime}(s)$. It can be easily shown that its parallel transport along $\gamma$ remains orthonormal, cf., \cite[p. 316]{bcs}. The geodesic $\gamma$ is minimal and has minimal length in its homotopy class. It follows that $L^{\prime\prime}(0)\geqslant 0$. Hence, for any piecewise smooth function $\phi$ with $\phi(0)=\phi(r)=0$, inserting $U:= \phi E_i$, for each $1\leqslant i\leqslant n-1$, into the second variation formula (2.5) yields
$$0\leqslant\int_0^r\big(g(\nabla_{\hat\gamma^{\prime}}(\phi E_i),\nabla_{\hat\gamma^{\prime}}(\phi E_i))-g(R(\phi E_i,\gamma^{\prime})\gamma^{\prime},\phi E_i)\big)ds.$$
This implies
\begin{align}\label{SVAL}
0\leqslant\ \sum\limits_{i=1}^{n-1}\int_0^r\big(g(\nabla_{\hat\gamma^{\prime}}(\phi E_i),\nabla_{\hat\gamma^{\prime}}(\phi E_i))-g(R(\phi E_i,\gamma^{\prime})\gamma^{\prime},\phi E_i)\big)ds.
\end{align}
Since $E_n=\gamma^{\prime}$ and $g(R(\phi E_n,\gamma^{\prime})\gamma^{\prime},\phi E_n)=0$, we have
\begin{align}\label{GRR}
\nonumber\sum\limits_{i=1}^{n-1}g(R(\phi E_i,\gamma^{\prime})\gamma^{\prime},\phi E_i)&=\sum\limits_{i=1}^{n}g(R(\phi E_i,\gamma^{\prime})\gamma^{\prime},\phi E_i)\\&=\phi^2{\gamma^{\prime}}^j{\gamma^{\prime}}^kR^i_{jik}=\phi^2\mathcal{R}ic(\gamma,\gamma^{\prime}).
\end{align}
Using $\nabla_{\gamma^{\prime}}(\phi E_i)=\frac{d\phi}{ds}E_i$ and substituting (\ref{GRR}) into (\ref{SVAL}), we get
\begin{align}\label{VF}
\nonumber 0\leqslant \ \int_0^r &\big(\sum\limits_{i=1}^{n-1}(\frac{d\phi}{ds})^2-\phi^2(s)\mathcal{R}ic(\gamma,\gamma^{\prime})\big)ds\\&=\int_0^r\big((n-1)(\frac{d\phi}{ds})^2-\phi^2(s)\mathcal{R}ic(\gamma,\gamma^{\prime})\big)ds.
\end{align}
By decomposition of (\ref{VF}) we have
\begin{align}\label{tajziye}
0\leqslant &\int_0^1(n-1)(\frac{d\phi}{ds})^2ds+\int_1^{r-1}(n-1)(\frac{d\phi}{ds})^2ds+\int_{r-1}^r(n-1)(\frac{d\phi}{ds})^2ds\nonumber\\&-\int_0^1\phi^2(s)\mathcal{R}ic(\gamma,\gamma^{\prime})\ ds-\int_{1}^{r-1}\phi^2(s)\mathcal{R}ic(\gamma,\gamma^{\prime})\ ds\nonumber\\&-\int_{r-1}^r\phi^2(s)\mathcal{R}ic(\gamma,\gamma^{\prime})\ ds.
\end{align}
Assuming the piecewise smooth function $\phi:[0,r]\longrightarrow [0,1]$ is given by
$$
\phi(s)=\left\{
\begin{array}{ccc}
s \qquad&  & 0\leqslant s\leqslant 1, \qquad\\
 1 \qquad&  & 1\leqslant s\leqslant r-1,\\
 r-s &  & r-1\leqslant s\leqslant r,
 \end{array}\right.
$$
reduces (\ref{tajziye}) to
\begin{align}\label{tajziye1}
\nonumber0\leqslant \ &2(n-1)-\int_0^1\phi^2(s)\mathcal{R}ic(\gamma,\gamma^{\prime})\ ds-\int_{1}^{r-1}\mathcal{R}ic(\gamma,\gamma^{\prime})\ ds\\&-\int_{r-1}^r\phi^2(s)\mathcal{R}ic(\gamma,\gamma^{\prime})\ ds.
\end{align}
Adding $\int_{0}^{r}\mathcal{R}ic(\gamma,\gamma^{\prime})ds$ to the both sides of (\ref{tajziye1}) yields
\begin{align}\label{tajziye2}
\nonumber\int_{0}^{r}\mathcal{R}ic(\gamma,\gamma^{\prime})\ ds\leqslant &\ 2(n-1)+\int_0^1(1-\phi^2(s))\mathcal{R}ic(\gamma,\gamma^{\prime})\ ds\\&+\int_{r-1}^r(1-\phi^2(s))\mathcal{R}ic(\gamma,\gamma^{\prime})\ ds.
\end{align}
By the minimizing property of $\gamma$, $d(p,\gamma(s))=s$. It follows that $d(p,\gamma(s))\leqslant 1$ for $0\leqslant s\leqslant 1$ and $\gamma(s)\in\mathcal{B}^+_p(1)$ for $0\leqslant s\leqslant 1$. Thus, $ \mathcal{R}ic(\gamma,\gamma^{\prime})\leqslant H_p $, where $0\leqslant s\leqslant 1$. Since $0\leqslant\phi\leqslant 1$ we have
\begin{align}\label{Hp}
\int_0^1(1-\phi^2(s))\mathcal{R}ic(\gamma,\gamma^{\prime})\ ds\leqslant H_p.
\end{align}
Similarly, the minimizing property of $\gamma$ yields $d(\gamma(s),q)=r-s$. Hence, $d(\gamma(s),q)\leqslant 1$,  for $r-1\leqslant s\leqslant r$ and consequently $ \mathcal{R}ic(\gamma,\gamma^{\prime})\leqslant H_q $. Therefore
\begin{align}\label{Hq}
\int_{r-1}^r(1-\phi^2(s))\mathcal{R}ic(\gamma,\gamma^{\prime})\ ds\leqslant H_q.
\end{align}
Replacing (\ref{Hp}) and (\ref{Hq}) in (\ref{tajziye2}) we conclude
$$\int_{0}^{r}\mathcal{R}ic(\gamma,\gamma^{\prime})\ ds\leqslant \ 2(n-1)+H_p+H_q.$$
As we have claimed.
\end{proof}
\section{Shrinking Ricci soliton Finsler spaces}
Let $(M,F)$ be a Finsler manifold and $V=v^{i} \frac{\partial}{\partial x^{i}}$
 a vector field on $M$. The triple $(M,F,V)$ is called  a \emph{Ricci soliton} if
 \begin{align}\label{Eq;DefRicciSoliton+}
2Ric_{jk}+\mathcal{L}_{\hat{V}}  {g_{jk}}=2 \lambda g_{jk},
\end{align}
where $\hat V$ is the complete lift of $V$ and  $\lambda \in {\mathbb{R}}$, cf. \cite{BY}.
By multiplying $y^jy^k$ in the both sides of (\ref{Eq;DefRicciSoliton+}), this equation leads to
\begin{align}
2F^2\mathcal{R}ic+\mathcal{L}_{\hat{V}}{F^2}=2 \lambda F^2.
\end{align}
  A Ricci soliton Finsler space is said to be \emph{shrinking, steady} or \emph{expanding} if $\lambda>0$, $\lambda=0$ or $\lambda<0$, respectively. A Ricci soliton Finsler space is called \emph{forward complete} (resp. \emph{backward complete}) if $ (M,F) $ is forward complete (resp. backward complete). For a vector field $ X=X^i\frac{\partial}{\partial x^i} $ on $ M $ define
\begin{align}
\Vert X\Vert_x=\max\limits_{y\in S_xM}\sqrt{g_{ij}(x,y)X^iX^j},
\end{align}
where $x\in M$, see \cite{bcs}, page 321. Since $ S_xM $ is compact, $ \Vert X\Vert_x $ is well defined.\\

\emph{Proof of Theorem \ref{theorem1}.}
Let $ p$ and $q $ be two points in $ M $ joined by a minimal geodesic $\gamma$ parameterized by the arc length $ s $, $ \gamma:[0,\infty)\longrightarrow M$. If $d(p,q)\leqslant 1$, then we have the assertion. Assume that $d(p,q)> 1$. Using (\ref{Eq;Lieder}) we have along $ \gamma $
\begin{align}\label{Eq;1}
\mathcal{L}_{\hat{V}}{F^2}={\gamma^{\prime}}^j{\gamma^{\prime}}^k\mathcal{L}_{\hat V}g_{jk}={\gamma^{\prime}}^j{\gamma^{\prime}}^k\big(\nabla_jV_k+\nabla_kV_j+2(\nabla_0V^l)C_{ljk}\big).
\end{align}
Since $ {\gamma^{\prime}}^j{\gamma^{\prime}}^k(\nabla_0V^l)C_{ljk}(\gamma(s),\gamma^{\prime}(s))=0 $, (\ref{Eq;1}) reduces to
\begin{align}\label{Eq;2}
\mathcal{L}_{\hat{V}}{F^2}=2{\gamma^{\prime}}^j{\gamma^{\prime}}^k\nabla_jV_k.
\end{align}
On the other hand, by metric compatibility of  Cartan connection,  along the geodesic $\gamma$, we have
\begin{align}\label{Eq;3}
{\gamma^{\prime}}^j{\gamma^{\prime}}^k\nabla_jV_k=
\nabla_{{\gamma^{\prime}}^j\frac{\delta}{\delta x^j}}({\gamma^{\prime}}^kV_k)=
\nabla_{\hat\gamma^{\prime}}({\gamma^{\prime}}^kV_k)=
\frac{d}{ds}({\gamma^{\prime}}^kV_k),
\end{align}
where, $\hat\gamma^{\prime}={\gamma^{\prime}}^j\frac{\delta}{\delta x^j}$. Replacing (\ref{Eq;3}) in (\ref{Eq;2}) we obtain
\begin{align}\label{Eq;4}
\mathcal{L}_{\hat{V}}{F^2}=2\frac{d}{ds}({\gamma^{\prime}}^kV_k).
\end{align}
In accordance with (\ref{Eq;a}) and (\ref{Eq;4}), along $\gamma$ we get
\begin{align*}
2F^2(\gamma,\gamma^{\prime})\mathcal{R}ic(\gamma,\gamma^{\prime})+2\frac{d}{ds}({\gamma^{\prime}}^kV_k)\geqslant 2 \lambda F^2(\gamma,\gamma^{\prime}).
\end{align*}
 Integrating  both sides of the last equation leads to
\begin{align}\label{UUU}
\int_0^r\mathcal{R}ic(\gamma,\gamma^{\prime})\ ds\geqslant \lambda r-{\gamma^{\prime}}^k(r)V_k+{\gamma^{\prime}}^k(0)V_k.
\end{align}
The Cauchy-Schwarz inequality gives $|{\gamma^{\prime}}^k(0)V_k|\leqslant \Vert V\Vert_p$ and $|{\gamma^{\prime}}^k(r)V_k|\leqslant \Vert V\Vert_q$.
Therefore, ${-\Vert V\Vert_p\leqslant\gamma^{\prime}}^k(0)V_k\leqslant \Vert V\Vert_p$ and ${-\Vert V\Vert_q\leqslant\gamma^{\prime}}^k(r)V_k\leqslant \Vert V\Vert_q$.
 Thus, we get
\begin{align}\label{Vert}
-{\gamma^{\prime}}^k(r)V_k+{\gamma^{\prime}}^k(0)V_k\geqslant -\Vert V\Vert_q-\Vert V\Vert_p.
\end{align}
Replacing (\ref{Vert}) in (\ref{UUU}) we have
\begin{align}\label{UUV}
\int_0^r\mathcal{R}ic(\gamma,\gamma^{\prime})\ ds\geqslant \lambda r-\Vert V\Vert_q-\Vert V\Vert_p.
\end{align}
According to (\ref{UUV}) and Lemma \ref{lemma} we have
\begin{align*}
2(n-1)+H_p+H_q\geqslant\lambda r-\Vert V\Vert_p-\Vert V\Vert_q.
\end{align*}
Finally,  for any $p,q\in M$, we get
  \begin{align*}
r=d(p,q)\leqslant\frac{1}{\lambda}\big(2(n-1)+H_p+H_q+\Vert V\Vert_p+\Vert V\Vert_q\big).
\end{align*}
Therefore, we have
\begin{align*}
  d(p,q)\leqslant \max\big\{1,\frac{1}{\lambda}\big(2(n-1)+H_p+H_q+\Vert V\Vert_p+\Vert V\Vert_q\big)\big\}.
  \end{align*}
This completes the proof of Theorem \ref{theorem1}.\hspace{\stretch{1}}$\Box$\\
If $M$ is a connected smooth manifold, then there exists a simply connected smooth manifold $\tilde M$, called its universal covering and a smooth covering map $p:\tilde M\longrightarrow M$ which is unique up to a diffeomorphism. The complete lift of $p$ is a map $\bar p:T\tilde M\longrightarrow TM$ given by
$$ \bar p(\tilde x,\tilde y)=(p(\tilde x),\tilde y^i\frac{\partial p}{\partial \tilde x^i})=(p(\tilde x),\tilde y^i\frac{\partial p^j}{\partial \tilde x^i}\frac{\partial}{\partial x^j}), $$
where $\tilde y\in T_{\tilde x}\tilde M$. A \textit{deck transformation} on the universal covering manifold $\tilde M$ is an isometry $ h:\tilde M\longrightarrow \tilde M$ such that $p\circ h=p$. The set of all deck transformations forms a group which is isomorphic to the fundamental group $\pi_1(M)$ of $M$.\\

\emph{Proof of Theorem \ref{thm2}.}
Let $(M,F)$ be a Finsler manifold and $ \tilde M $ its universal covering with the smooth covering map $ p:\tilde M\longrightarrow M $.
 Pull back of the complete lift $ \bar p $ namely,  $\bar p^*F:=F\circ \bar p:T\tilde M\longrightarrow [0,\infty) $ defines a Finsler structure on $ \tilde M $. In fact, one can easily check out the three conditions of Finsler structure, as follows.
  The regularity condition is satisfied since $ F $ and $ p $ are $ C^{\infty} $, and so is $\bar p^*F $.
\begin{align*}
\bar p^*F(x,\lambda y)&=F\circ \bar p(x,\lambda y)=F(p(x),\lambda y^i\frac{\partial p}{\partial x^i})\\&=\lambda F(p(x),y^i\frac{\partial p}{\partial x^i})=\lambda\bar p^*F(x,y).
\end{align*}
From which positive homogeneity is deduced. Next, assuming $\bar p^*x^i=\tilde x^i $ and $\bar p^*y^i=\tilde y^i $, for strong convexity we have
\begin{align*}
\tilde g_{ij}:=[\frac{1}{2}(\bar p^*F)^2]_{\tilde y^i\tilde y^j}=\frac{1}{2}\frac{\partial^2((\bar p^*F)^2)}{\partial \tilde y^i\partial \tilde y^j}=\frac{1}{2}\frac{\partial^2(\bar p^*F^2)}{\partial \tilde y^i\partial \tilde y^j}.
\end{align*}
One can easily check that
$$\frac{\partial(\bar p^*F^2)}{\partial \tilde y^i}=\bar p^*\frac{\partial F^2}{\partial y^i},$$
therefore,
\begin{align}\label{Eq:pd}
\tilde g_{ij}=[\frac{1}{2}(\bar p^*F)^2]_{\tilde y^i\tilde y^j}=\frac{1}{2}\frac{\partial^2(\bar p^*F^2)}{\partial \tilde y^i\partial \tilde y^j}=\bar p^*[\frac{1}{2}F^2]_{ y^i y^j}=\bar p^*g_{ij}.
\end{align}
Using the facts that $[\frac{1}{2}F^2]_{ y^i y^j}$ is positive definite on $TM_0$ and $\bar p^* $ is a local diffeomorphism ($p$ is the smooth covering map), $ \bar p^*[\frac{1}{2}F^2]_{ y^i y^j}$ is also positive definite on $T\tilde M_0$ and hence $\tilde F:=\bar p^*F$, defines a Finsler structure on $T\tilde M_0$ moreover $ (\tilde M,\tilde F)$ is locally isometric to $ (M,F) $.
 Here, every geometric object defined on $(\tilde M,\tilde F)$ are decorated by a tilde. By means of the local isometry
 $p:(\tilde M,\tilde F)\longrightarrow (M,F),$ and the inequality (\ref{Eq;a}),  we have
\begin{align*}
\bar p^*(2F^2\mathcal{R}ic+\mathcal{L}_{\hat{V}}{F^2})\geqslant 2 \bar p^*(\lambda F^2).
\end{align*}
Linearity of $\bar{p}^*$ implies
\begin{align}\label{ineqaulity}
2(\bar p^*F^2)( \bar p^*{\mathcal{R}ic})+\bar p^*\mathcal{L}_{\hat{V}}F^2\geqslant 2\lambda \bar p^*( F^2).
\end{align}
Let's denote $ W=p^*V $, then by virtue of (\ref{Eq:pd}),  properties of Lie  derivative  and the pull back $\bar{p}^*$ we obtain
\begin{align}\label{Eq;5}
 \bar p^*\mathcal{L}_{\hat{V}}  {F^2}=\mathcal{L}_{\hat{W}}  {\tilde F^2}.
\end{align}
On the other hand, one can easily show that $ \bar p^*\mathcal{R}ic=\tilde{\mathcal{R}}ic $, see \cite{BY}.
Replacing, $\bar p^*\mathcal{R}ic=\tilde{\mathcal{R}}ic$, $\tilde F=\bar p^*F$ and (\ref{Eq;5})  in (\ref{ineqaulity}), leads
\begin{align*}
2\tilde F^2\tilde{\mathcal{R}}ic+\mathcal{L}_{\hat{W}}{\tilde F^2}\geqslant 2\lambda \tilde  F^2.
\end{align*}
Our universal covering $(\tilde M,\tilde F)$ is  geodesically complete because $(M,F)$ is so. In fact, let $\tilde{\gamma}(t)$ be any geodesic emanating from some point $\tilde x\in \tilde M$ at $t=0$.
Clearly,  $\gamma(t):=p(\tilde{\gamma}(t))$ its image by the local isometry  $p$ is also a geodesic. Upon geodesically completeness assumption, $\gamma(t)$ is extendible to all $t\in [0,\infty)$. The said local isometry now implies the same for $\tilde{\gamma}(t)$. Hence the universal covering $(\tilde M,\tilde F)$ is geodesically complete.

 Let $h$ be a deck transformation on $\tilde M$ and $\tilde x\in\tilde M $.
 By definition, $h$ is an isometry, and the forward metric balls $\mathcal{B}^+_{\tilde x}(1)$ and $\mathcal{B}^+_{h(\tilde x)}(1)$ are isometric. Therefore $H_{\tilde x}=H_{h(\tilde x)}$ and $ \Vert W\Vert_{\tilde x}=\Vert W\Vert_{h(\tilde x)}$. By relation (\ref{theorem1bound}) in the proof of  Theorem \ref{theorem1}, for the points $\tilde x$ and $h(\tilde x)$ we get
\begin{align*}
  d(\tilde x,h(\tilde x))\leqslant &\max\big\{1,\frac{1}{\lambda}\Big(2(n-1)+H_{\tilde x}+H_{h(\tilde x)}+\Vert W\Vert_{\tilde x}+\Vert W\Vert_{h(\tilde x)}\Big)\big\}\\&= \max\big\{1,\frac{2}{\lambda}\big(n-1+H_{\tilde x}+\Vert W\Vert_{\tilde x}\big)\big\},
  \end{align*}
for any deck transformation $h$. Thus $p^{-1}(x)$, where $ x=p(\tilde x)$, is forward bounded. Using the forward geodesic completeness and the Hopf-Rinow theorem, the closed and forward bounded subset $ p^{-1}(x) $ of $\tilde M$ is compact and being discrete is finite. By assumption, $ M $ is connected, so all of its fundamental groups, namely $ \pi_1(M,x) $ are isomorphic, where $ x $ denotes the base point. Since $\tilde M $ is a universal covering, $\pi_1(M,x)$ is in a bijective correspondence with $ p^{-1}(x) $ and therefore $\pi_1(M)$ is finite and the first de Rham cohomology group vanishes, i.e., $H^1_{\mathrm{dR}}(M)=0$.
This completes the proof of Theorem \ref{thm2}.\hspace{\stretch{1}}$\Box$\\
\begin{cor}
Let $(M,F,V)$ be a forward and backward complete shrinking Ricci soliton. Then the fundamental group $\pi_1(M)$ of $M$ is finite and therefore $H^1_{\mathrm{dR}}(M)=0$.
\end{cor}

Using this Theorem, we have the following result for a fundamental group of the sphere bundle $SM$.
\begin{cor}\label{thm3}
Let $(M,F,V)$, for $n\geq 3$, be a forward and backward complete shrinking Ricci soliton Finsler space. Then the fundamental group $\pi_1(SM)$ of $SM$ is finite and therefore $H^1_{\mathrm{dR}}(SM)=0$.
\end{cor}
\begin{proof}
Let $ \tilde M $ be the universal covering manifold of $ M $ with the smooth covering map $ p:\tilde M\longrightarrow M $. It is well known that the homotopy sequence of the fibre (sphere) bundle $(S\tilde M,\tilde\pi,\tilde M,S^{n-1})$, namely
\begin{align}\label{sequence}
\cdots\longrightarrow\pi_1(S^{n-1})\longrightarrow\pi_1(S\tilde M)\longrightarrow\pi_1(\tilde M)\longrightarrow\cdots,\end{align}
is exact. Since $\tilde M$ is simply connected, $\pi_1(\tilde M)=0$. We know that $\pi_1(S^{n-1})=0$. Thus  (\ref{sequence}) implies  $\pi_1(S\tilde M)=0$. One can easily check that $\bar p:S\tilde M\longrightarrow SM $ is a smooth covering map.
 Therefore, $ S\tilde M $ is the universal covering manifold of $ SM $. For every section $y$ of  $ SM$, there is a point $x\in M$ such that $y\in S_xM$.
 According to the proof of Theorem \ref{thm2}, $p^{-1}(x)$ is a finite set and consequently,  ${\bar p}^{-1}(y)\subseteq\bigcup\limits_{\tilde x\in{p}^{-1}(x)}S_{\tilde x}\tilde M$  is compact and being discrete is finite. Thus the fundamental group $\pi_1(SM)$ is finite and therefore $H^1_{\mathrm{dR}}(SM)=0$.
\end{proof}
{\bf Acknowledgment}\\
The second author would like to thank the IPM School of Mathematics for a partial support.

\date{ \small Behroz Bidabad and Mohamad Yar Ahmadi\\
Amirkabir University of Technology (Tehran Polytechnic)\\
424 Hafez Ave. 15914 Tehran, Iran.\\
E-mails: bidabad@aut.ac.ir; m.yarahmadi@aut.ac.ir}
\end{document}